\documentclass[centertags,12pt]{amsart}

\usepackage{latexsym}
\usepackage{amssymb}  

\numberwithin{equation}{section}

\makeatletter
\renewcommand{\subsection}{\@startsection
{subsection}{2}{0mm}{\baselineskip}{-0.25cm}
{\normalfont\normalsize\em}}
\makeatother

\newtheorem{theorem}{Theorem}[section]
\newtheorem{proposition}[theorem]{Proposition}
\newtheorem{corollary}[theorem]{Corollary}
\newtheorem{lemma}[theorem]{Lemma}

{\theoremstyle{definition}

\newtheorem{example}[theorem]{Example}}

\theoremstyle{remark}
\newtheorem*{claim*}{Claim}
\newtheorem{remark}[theorem]{Remark}

\def\bN{\mathbb N}
\def\bP{\mathbb P}

\def\F{\mathbf F}

\def\K{\mathbf K}
\def\M{\mathbf M}

\def\cD{\mathcal D}
\def\cE{\mathcal E}

\def\cH{\mathcal H}
\def\cJ{\mathcal J}

\def\cR{\mathcal R}

\def\cX{\mathcal X}

\def\deg{{\rm deg}}

\def\Fl{{\mathbb F_\ell}}
\def\F49{{\mathbb F_{49}}}
\def\F64{{\mathbb F_{64}}}
\def\Fq2{{\mathbb F_{q^2}}}

\begin{document}

\author[N.\, Arakelian]{Nazar Arakelian}
\author[S.\, Tafazolian]{Saeed Tafazolian}
\author[F.\, Torres]{Fernando Torres} 


\thanks{\today}

\title[The spectrum for the genera of maximal curves]{On 
the spectrum for the genera of maximal curves over small 
fields}

  \address{CMCC/Universidade Federal do ABC, Avenida dos Estados 5001, 
  09210-580, Santo Andr\'e, SP-Brasil} 
  \email{n.arakelian@ufabc.edu.br}

  \address{School of Mathematics,
  Institute for Research in Fundamental Science (IPM),
  P.O. Box 19395-5746,
  Tehran, Iran\\
  Dept. of Mathematics and Computer Scinence,
  Amirkabir University of Technology,
  424 Hafez Ave, 
  Tel: +98 (21) 64540
  P.O. Box: 15875-4413, Tehran, Iran}
  \email{saeed@gmail.com}

  \address{IMECC/UNICAMP, R. S\'ergio Buarque de Holanda 651, 
Cidade Universit\'aria \lq\lq Zeferino Vaz", 13083-859, 
Campinas, SP, Brazil}

  \email{ftorres@ime.unicamp.br}

    \begin{abstract} Motivated by previous computations in 
   Garcia, Stichtenoth and Xing (2000) paper \cite{GSX}, we 
   discuss the spectrum $\M(q^2)$ for the genera of maximal curves 
   over finite fields of order $q^2$ with $7\leq q\leq 16$. 
   In particular, by using a result in Kudo and Harashita 
   (2016) paper \cite{KH}, the set $\M(7^2)$ is completely
   determined.
    \end{abstract}

  \maketitle
  
    \section{Introduction}\label{s1}
    
Let $\cX$ be a (projective, nonsingular, geometrically 
irreducible, algebraic) curve of genus $g$ defined over a 
finite field $\K=\Fl$ of order $\ell$. The following 
inequality is the so-called {\em Hasse-Weil bound} on 
the size $N$ of the set $\cX(\K)$ of $\K$-rational points of $\cX$:
   \begin{equation}\label{eq1.1}
   |N-(\ell+1)|\leq 2g\cdot \sqrt{\ell}\, .
   \end{equation}
 In Coding Theory, Cryptography, or Finite Geometry one is often 
 interested in curves with 
   \lq\lq many points", namely those with $N$ as bigger as 
   possible. In this paper, we work out over fields of square 
   order, $\ell=q^2$, and deal with so-called 
   {\em maximal curves 
   over $\K$}; that is to say, those curves attained the upper bound in 
   (\ref{eq1.1}), namely
   \begin{equation}\label{eq1.2}
   N=q^2+1+2g\cdot q\, .
   \end{equation}
     The subject matter of this note is in fact concerning 
     the {\em spectrum for the genera} of maximal curves over $\K$,
  \begin{equation}\label{eq1.3}
\M(q^2):=\{g\in\bN_0:\text{there is a maximal curve over $\K$ of 
genus $g$}\}\, .
  \end{equation}  
In Section \ref{s2} we subsume basic facts on a maximal curve 
$\cX$ being the key property the existence of a very ample 
linear series $\cD$ on $\cX$ equipped with a nice property, namely 
(\ref{eq2.2}). In particular, Castelnuovo's genus bound 
(\ref{eq2.3}) and Halphen's theorem imply a 
nontrivial restriction on the genus $g$ of $\cX$, stated in 
(\ref{eq3.1}) (see \cite{KT2}) and thus $g\leq q(q-1)/2$ (Ihara's 
bound \cite{I}). 

Let $r$ be the dimension of $\cD$. Then $r\geq 2$ by 
(\ref{eq2.2}), and the condition $r=2$ is equivalent to 
$g=q(q-1)/2$, or equivalent to $\cX$ being $\K$-isomorphic to 
the Hermitian curve $y^{q+1}=x^q+x$ \cite{X-Sti}, \cite{FT}.  
Under certain conditions, we have a similar result for $r=3$ 
in Corollary \ref{cor2.1} and Proposition \ref{prop3.1}. In 
fact, in Section \ref{s3} we bound $g$ via St\"ohr-Voloch 
theory \cite{SV} applied to $\cD$ being the main results the 
aforementioned proposition and its Corollary \ref{cor3.1}. 
Finally, in Section \ref{s4} we apply all these results toward 
the computation of $\M(q^2)$ for $q=7,8,9,11,13,16$. In fact, 
here we improve \cite[Sect. 6]{GSX} and, in particular, we 
can compute $\M(7^2)$ (see Corollary \ref{cor4.1}) by using 
Corollary \ref{cor3.1} and a result of Kudo and Harashita 
\cite{KH} which asserts that there is no maximal curve of 
genus 4 over ${\mathbb F}_{49}$.

{\bf Conventions.} $\bP^s$ is the $s$-dimensional projective 
space defined over the algebraic closure of the base field.

      \section{Basic Facts on Maximal curves}\label{s2}
   
   Throughout, let $\cX$ be a maximal curve over the field 
   $\K=\Fq2$ of order $q^2$ of 
   genus $g$. Let $\Phi:\cX\to\cX$ be the Frobenius 
   morphism relative to $\K$ (in particular, the set of fixed 
   points of $\Phi$ coincides with $\cX(\K)$). For a fixed 
   point $P_0\in\cX(\K)$, let $j:\cX\to \cJ, P\mapsto 
   [P-P_0]$ be the embedding of $\cX$ into its Jacobian 
   variety $\cJ$.  Then, in a natural way, $\Phi$ induces a 
   morphism $\tilde \Phi:\cJ\to\cJ$ such that  
     \begin{equation}\label{eq2.1}
   j\circ \tilde \Phi = \tilde \Phi\circ j\, .
     \end{equation}
   Now from (\ref{eq1.2}) the enumerator of the Zeta Function of 
   $\cX$ is given by the polynomial $L(t)=(1+qt)^{2g}$. It turns 
   out that $h(t):=t^{2g}L(t^{-1})$ is the characteristic 
   polynomial of $\tilde \Phi$; i.e., $h(\tilde \Phi)=0$ on 
   $\cJ$. As a matter of fact, since $\tilde\Phi$ is 
   semisimple and the representation of endomorphisms of $\cJ$ on 
   the Tate module is faithful, 
   from (\ref{eq2.1}) it follows that
     \begin{equation}\label{eq2.2}
     (q+1)P_0\sim qP+\Phi(P)\, ,\quad P\in\cX\, .
     \end{equation}
This suggests to study the {\em Frobenius linear series} on 
$\cX$, namely the complete linear series 
$\cD:=|(q+1)P_0|$ which is in fact a $\K$-invariant of $\cX$ by 
(\ref{eq2.2}); see \cite{FGT}, \cite[Ch. 10]{HKT} 
for further information. 

Moreover, $\cD$ is a very ample linear series in the 
following sense. Let $r$ be the dimension of $\cD$, which we 
refer as the {\em Frobenius dimension} of $\cX$, and 
$\pi:\cX\to \bP^r$ be a morphism related to $\cD$; we notice 
that $r\geq 2$ by (\ref{eq2.2}). Then $\pi$ is an embedding 
\cite[Thm. 2.5]{KT1}. In particular, Castelnuovo's genus 
bound applied to $\pi(\cX)$ gives the following constrain 
involving the genus $g$ and Castelnuovo numbers $c_0(r,q+1)$:
    \begin{equation}\label{eq2.3}
g\leq c_0(r)=c_0(r,q+1):=\begin{cases} 
((2q-(r-1))^2-1)/8(r-1) & \text{if $r$ is even}\, ,\\
(2q-(r-1))^2/8(r-1) & \text{if $r$ is odd}\, .
  \end{cases}
  \end{equation}
  \begin{remark}\label{rem2.1} A direct computation shows that 
  $c_0(r)\leq c_0(s)$ provided that $r\geq s$. 
  \end{remark}
  Since $c_0(r)\leq c_0(2)=q(q-1)/2$, as $r\geq 2$, then 
  $g\leq q(q-1)/2$ which is a well-known fact on maximal curves 
  over $\K$ due to Ihara \cite{I}. In addition, $c_0(r)\leq c_0(3)=
  (q-1)^2/4$ for $r\geq 3$, so that the genus $g$ of 
  a maximal curve 
  over $\K$ does satisfy the following condition (see \cite{FT})
    \begin{equation}\label{eq2.4}  
  g\leq c_0(3)=(q-1)^2/4\quad\text{or}\quad
  g=c_0(2)=q(q-1)/2\, .
  \end{equation}
As a matter of fact, the following 
  sentences are equivalent.
  \begin{lemma}\label{lemma2.1}{\rm (\cite{R-Sti}, \cite{FT})}
  \begin{enumerate}
  \item[\rm(1)]\quad $g=c_0(2)=q(q-1)/2;$
  \item[\rm(2)]\quad $(q-1)^2/4< g\leq q(q-1)/2;$
  \item[\rm(3)]\quad $r=2;$
  \item[\rm(4)]\quad $\cX$ is $\K$-isomorphic to the Hermitian 
  curve $\cH:\, y^{q+1}=x^q+x.$
  \end{enumerate}
  \end{lemma}
    \begin{corollary}\label{cor2.1} Let $\cX$ be a maximal curve 
    over $\K$ of genus $g$ and Frobenius dimension $r.$ Suppose that
  $$
c_0(4)=(q-1)(q-2)/6<g\leq c_0(3)=(q-1)^2/4\, .
  $$
  Then $r=3.$ 
  \end{corollary}
  \begin{proof} If $r\geq 4$, then $g\leq  (q-1)(q-2)/6$ 
  by (\ref{eq2.3}); so $r=2$ or $r=3$. Thus $r=3$ by Lemma 
  \ref{lemma2.1} and hypothesis on $g$.
  \end{proof}
  Under certain conditions, this result will be improved 
  in Proposition \ref{prop3.1}. 
  
The following important remark is commonly attributed to 
J.P. Serre. 
  \begin{remark}\label{rem2.2} Any curve (nontrivially) 
  $\K$-covered by 
  a maximal curve over $\K$ is also maximal over $\K$. In 
  particular, any 
  subcover over $\K$ of the Hermitian curve is so; see e.g. 
  \cite{GSX}, \cite{CKT}.
  \end{remark}
  \begin{remark}\label{rem2.3} We do point out that there are 
  maximal curves over $\K$ 
  which cannot be (nontrivially) $\K$-covered by the Hermitian 
  curve $\cH$, see 
  \cite{GK}, \cite{T^3}, \cite{gqz}. \\
  We also notice that there are maximal curves over $\K$ that 
  cannot be Galois covered by the Hermitian curve \cite{G-Sti}, 
  \cite{DM}, \cite{T^3}, \cite{gqz}. \\
  We also observe that all the examples occuring in this 
  remark are defined over fields of order $q^2=\ell^6$ with $\ell>2$.
  \end{remark}
  
  \section{The set $\M(q^2)$}\label{s3}
  
In this section we investigate the spectrum $\M(q^2)$ 
for the genera of maximal curves defined in 
(\ref{eq1.3}). By using Remark \ref{rem2.2} this set has already been 
computed for $q\leq 5$ \cite[Sect. 6]{GSX}. 
As a matter of fact, $\M(2^2)=\{0,1\}$, 
$\M(3^2)=\{0,1,3\}$, $\M(4^2)=\{0,1,2, 6\}$, and $\M(5^2)=
\{0,1,2,3,4,10\}$. Thus from now on we assume $q\geq 7$.

Let $c_0(r)$ be the Castelnuovo's number in (\ref{eq2.3}) and  
$g\in \M(q^2)$. 
It is known that $g=\lfloor c_0(3)\rfloor$ if and only if $\cX$ is 
the quotient 
of the Hermitian curve $\cH$ by certain involution 
\cite{FGT}, \cite{AT}, \cite{KT2}. Indeed, $\cX$ is uniquely determined by plane models of type: $y^{(q+1)/2}=x^q+x$ if $q$ 
is odd, and $y^{q+1}=x^{q/2}+\ldots+x$ otherwise. 

Let us consider next an improvement on (\ref{eq2.4}). 
If $r\geq 4$, from (\ref{eq2.3}), $g\leq c_0(4)=\\
(q-1)(q-2)/6$. 
Let $r=3$ and suppose that
  $$
c_1(3)=c_1(q^2,3):=(q^2-q+4)/6 <g \leq c_0(3).
  $$
Here Halphen's theorem implies that $\cX$ is contained in a 
quadric surface and so $g=c_0(3)$ (see \cite{KT2}). 
In particular, (\ref{eq2.4}) improves to
   \begin{equation}\label{eq3.1}  
g\leq c_1(3)\, ,\quad\text{or}\quad g=\lfloor 
c_0(3)\rfloor\, ,\quad\text{or}\quad
g=c_0(2)\, .
   \end{equation}
  Next we complement Corollary \ref{cor2.1} under certain 
  extra conditions.
    \begin{proposition}\label{prop3.1} Let $\cX$ be a maximal 
    curve over $\K$, $q\not\equiv 0\pmod{3}$, of genus 
  $g$ with Frobenius dimension $r=3$ such that 
$(4q-1)(2g-2)>(q+1)(q^2-5q-2).$ Then 
  $$ 
g\geq c_0(4)+(q+1)/6= (q^2-2q+3)/6\, . 
   $$  
  \end{proposition}
  \begin{proof} We shall apply St\"ohr-Voloch theory \cite{SV} to 
  $\cD=|(q+1)P_0|$. 
Let $R=\sum_P v_P(R)P$ and $S=\sum_Pv_P(S)P$ denote respectively 
the ramification and Frobenius divisor of 
$\cD$. Associated to each point $P\in\cX$, there is a sequence of 
the possible intersection multiplicities of $\cX$ with hyperplanes 
in $\bP^3$, namely $\cR(P): 
0=j_0(P)<1=j_1(P)<j_2(P)<j_3(P)$. From (\ref{eq2.2}), 
$j_3(P)=q+1$ 
(resp. $j_3(P)=q$) if $P\in\cX(\Fq2)$ (resp. $P\not\in\cX(\Fq2$)). 
Moreover, 
the sequence $\cR(P)$ is the same for all but a finitely number of points (the 
so-called $\cD$-Weierstrass points of $\cX$); such a sequence (
the {\em orders} of $\cD$) will be denoted by 
$\cE: 0=\epsilon_0<1=\epsilon_1<\epsilon_2<\epsilon_3=q$. One can show that the 
numbers $1=\nu_1<q=\nu_2$ (the {\em 
$\K$-Frobenius orders} of $\cD$) satisfy the 
very basic properties (5) and (6) below (cf. \cite{SV}):
  \begin{enumerate}
\item[\rm(1)] \quad $j_i(P)\geq \epsilon_i$ for any $i$ and $P\in\cX$;
\item[\rm(2)] \quad $v_P(R)\geq 1$ for $P\in \cX(K)$;
\item[\rm(3)] \quad $\deg(R)=(\epsilon_3+\epsilon_2+1)(2g-2)+(r+1)(q+1)$;
\item[\rm(4)] \quad ($p$-adic criterion) If $\epsilon$ is an order and 
$\binom{\epsilon}{\eta}\not\equiv 0\pmod{p}$, then $\eta$ is also an order;
\item[\rm(5)] \quad $v_P(S)\geq j_2(P)+(j_3(P)-\nu_2)=j_2(P)+1$ for $P\in \cX(K)$;
\item[\rm(6)]\quad $\deg(S)=(\nu_1+\nu_2)(2g-2)+(q^2+r)(q+1)$.
   \end{enumerate}
{\bf Claim $\epsilon_2=2$.} Suppose that $\epsilon_2\geq 3$; 
then $\epsilon_2\geq 
4$ by the $p$-adic criterion. Then the maximality of $\cX$ gives
  $$
\deg(S)=(1+q)(2g-2)+(q^2+3)(q+1)\geq 5(q+1)^2+5q(2g-2)
  $$
so that 
  $$
(q+1)(q^2-5q-2)\geq (4q-1)(2g-2)\, ,
  $$
a contradiction and the proof of the claim follows. 

Finally, we use the ramification divisor $R$ of $\cD$; we have
  $$
\deg(R)=(q+2+1)(2g-2)+4(q+1)\geq (q+1)^2+q(2g-2)
  $$
and thus $g\geq (q^2-2q+3)/6$.
  \end{proof}  
  \begin{corollary}\label{cor3.1} Let $\cX$ be a maximal curve 
  over $\K,$ of genus $g,$ where $q\not\equiv 0\pmod{3}.$ Then
   $$
\text{$g\geq (q^2-2q+3)/6$\quad provided that 
$g>(q-1)(q-2)/6\, .$}
  $$
   \end{corollary}  
   \begin{proof} Let $\cD$ be the Frobenius linar series of $\cX$ 
   and $r$ the Frobenius dimension. By (\ref{eq2.3}) and Lemma 
   \ref{lemma2.1}, we can assume $r=3$. Now the 
   hypothesis on $g$ is 
   equivalent to \\
   $(2g-2)>(q+1)(q-4)/3$; thus 
  $$
(4q-1)(2g-2)>(4q-1)(q+1)(q-4)/3>(q+1)(q^2-5q-2)\, ,
  $$
and the result follows from Proposition \ref{prop3.1}.   
   \end{proof}
   
   \section{$\M(q^2)$ for $7\leq q\leq 16$}\label{s4}
   
In this section we shall improve on the following 
computations which follow from \cite[Remark 6.1]{GSX} and (\ref{eq3.1}).
  \begin{proposition}\label{prop4.1}
  \begin{enumerate}
  \item[\rm(1)] $\{0,1,2,3,5,7,9,21\}\subseteq 
  \M(7^2)\subseteq [0,7]\cup\{9\}\cup\{21\};$
   \item[\rm(2)] $\{0,1,2,3,4, 6,7,9,10, 12,28\}\subseteq 
   \M(8^2)\subseteq [0,10]\cup\{12\}\cup\{28\};$
\item[\rm(3)] $\{0,1,2,3,4,6,8,9,12,16,36\}\subseteq 
\M(9^2)\subseteq [0,12]\cup\{16\}\cup\{36\};$
\item[\rm(4)] $\{0,1,2,3,4,5,7,9,10,11,13,15,18,19,25,55\} 
\subseteq \M({11}^2)
\subseteq [0,19]\cup\{25\}\cup\{55\};$
\item[\rm(5)] $\{0,2,3,6,9,12,15,18,26,36,78\}
\subseteq \M({13}^2)\subseteq [0,26]\cup\{36\}\cup\{78\};$
\item[\rm(6)] $\{0,1,2,4,6,8,12,24,28,40,56,120\}\subseteq 
\M({16}^2)\subseteq [0,40]\cup\{56\}\cup\{120\}.$
   \end{enumerate}
   \end{proposition}
   \begin{proposition}\label{prop4.2} Let $\M(q^2)$ be 
  the spectrum for the genera of maximal curves over $\K.$ Then
   \begin{enumerate}
   \item[\rm(1)] $6\not\in\M(7^2);$
   \item[\rm(2)] $8 \not\in\M(8^2);$
   \item[\rm(3)] $16\not\in\M({11}^2);$
   \item[\rm(4)] $23,24 \not\in\M({13}^2);$
   \item[\rm(5)] $36,37\not\in\M({16}^2).$
   \end{enumerate}   
   \end{proposition}
   \begin{proof} Let $q=7$. By Corollary \ref{cor3.1}, 
   $g=6\not\in \M(7^2)$. The other cases are handle in a 
   similar way.
   \end{proof}
   \begin{corollary}\label{cor4.1} We have
   $$
   \M(7^2)=\{0,1,2,3,5,7,9,21\}\, .
   $$
  \end{corollary}
  \begin{proof} By the above Propositions, it is enough to show 
  that $4\not\in \M(7^2)$. Indeed, this is the case as 
  follows from a result in Kudo and Harashita paper 
  \cite[Thm. B]{KH} concerning superspecial curves.
  \end{proof}
 \begin{remark}\label{rem4.1} To compute $\M(q^2)$ for $q=8,9,11,13,16$ we need to answer the following questions: 
 \begin{enumerate}
 \item[\rm(1)] Is $5\in\M(8^2)$?
 \item[\rm(2)] Are $5, 7,10,11\in\M(9^2)$? 
 \item[\rm(3)] Are $8,12,14,17\in \M({11}^2)$?
 \item[\rm(4)] Are $1,4,5,7,8,10,11,13,14, 16,17,19,20,21,22\in
 \M({13}^2)$?  
 \item[\rm(5)] Are $3,5,7,9,10, 11, 13,14,\ldots, 22,23,25,26, 27, 
 29,30,31,32,33,34, 
 35, 38, 39\in \M({16}^2)$?
   \end{enumerate}
   \end{remark}
      \begin{example}\label{ex4.1} Here, for the sake of 
completeness, we provide an example of a maximal curve 
of genus $g$ for each $g\in \M(7^2)$; cf. \cite{many}, \cite{TT}.
  \begin{enumerate}
  \item[\rm(1)] ($g=0$) The rational curve;
  \item[\rm(2)] ($g=1$) $y^2=x^3+x$;
  \item[\rm(3)] ($g=2$) $y^2=x^5+x$;
  \item[\rm(4)] ($g=3$) $y^2=x^7+x$;
  \item[\rm(5)] ($g=5$) $y^8=x^4-x^2$;
  \item[\rm(6)] ($g=7$) $y^{16}=x^9-x^{10}$;
  \item[\rm(7)] ($g=9$) $y^4=x^7+x$;
  \item[\rm(8)] ($g=21$) $y^8=x^7+x$.
   \end{enumerate}
   \end{example}
   \begin{remark}\label{rem4.2} The curves in (6), (7), and (8) 
   above are unique up to $\mathbb F_{49}$-isomorphism; see respectively 
   \cite{FGP}, \cite{FGT}, and \cite{R-Sti}.
    \end{remark}
     
   {\bf Acknowledgment.} The first author was partially supported by 
   FAPESP, grant 2013/00564-1. The second author was in part supported 
   by a grant from IPM (No. 93140117). The third author was partially 
   supported by CNPq (308326/2014-8).

   \end{document}